\documentclass[a4paper,oneside,english]{amsart}
\usepackage[T1]{fontenc}
\usepackage[latin9]{inputenc}
\usepackage{babel}
\usepackage{amstext}
\usepackage{amsthm}
\usepackage{xcolor}
\usepackage{amssymb}
\usepackage{setspace}
\usepackage[hidelinks,colorlinks=true,linkcolor=blue,citecolor=blue]{hyperref}

\makeatletter

\pdfpageheight\paperheight
\pdfpagewidth\paperwidth

\numberwithin{equation}{section}
\numberwithin{figure}{section}
\theoremstyle{plain}
\newtheorem{thm}{\protect\theoremname}
\theoremstyle{remark}
\newtheorem{rem}[thm]{\protect\remarkname}
\theoremstyle{plain}
\newtheorem{cor}[thm]{\protect\corollaryname}
\theoremstyle{plain}
\newtheorem*{thm*}{\protect\theoremname}

\def\G{{\mathbb{G}}}
\def\R{{\mathbb{R}}}
\def\d{{w(x)}}
\def\dd{{\text{dist}(x,\partial \mathbb{G}^{+})}}
\def\dx{{\text{dist}(\cdot,\partial \mathbb{G}^{+})}}

\makeatother

\providecommand{\corollaryname}{Corollary}
\providecommand{\remarkname}{Remark}
\providecommand{\theoremname}{Theorem}

\begin{document}
\title[logarithmic Hardy and Hardy-Poincar\'e inequalities]{Geometric logarithmic Hardy and  Hardy-Poincar\'e inequalities on  stratified groups}

 \author[M. Chatzakou]{Marianna Chatzakou}
\address{
	Marianna Chatzakou:
	\endgraf
	Department of Mathematics: Analysis, Logic and Discrete Mathematics
	\endgraf
	Ghent University, Krijgslaan 281, Building S8, B 9000 Ghent
	\endgraf
	Belgium
	\endgraf
	{\it E-mail address} {\rm marianna.chatzakou@ugent.be}}

\begin{abstract}
We develop a unified strategy to obtain the geometric logarithmic Hardy inequality on any open set $M\subset \G$ of a stratified group $\G$, provided the validity of the Hardy inequality in this setting, where the so-called ``weight'' is regarded to be any measurable non-negative function $w$ on $M$. Provided the legitimacy of the latter for some $M,w$, we also show an inequality that is an extension of the ``generalised Poincar\'e inequality'' introduced by Beckner with the addition of the weight $w$, and this is referred to as the ``geometric Hardy-Poincar\'e inequality''. The aforesaid inequalities become explicit in the case where $M=\G^{+}$, the half-space of $\G$, when $w(\cdot)=\dx$, and in the case where $M=\G$, when $w$ is the ``horizontal norm'' on the first stratum of $\G$. For the second case, the semi-Gaussian analogue of the derived inequalities is proved, when the Gaussian measure is regarded with respect to the first stratum of $\G$. Applying our results to the case where $\G=\R^n$ (abelian case) we  generalise the classical probabilistic Poincar\'e inequality  by adding weights.
\end{abstract}


\maketitle

\section{Introduction}
\subsection{Hardy and logarithmic Hardy inequalities}
In the Euclidean setting, the classical Hardy inequality asserts that for any $p>1$, and for $u \in C_{0}^{\infty}(\Omega \setminus \{0\})$, where $\Omega \subset \R^N$ is a domain containing the origin, we have:
\begin{equation}\label{in.1}
    \int_{\R^N}|\nabla u(x)|^p\,dx \geq \left| \frac{N-p}{p}\right|^p \int_{\R^N}\frac{|u(x)|^p}{|x|^p}\,dx\,,
\end{equation}
where $|x|$ stands for the Euclidean norm of $x \in \R^N$, and the constant $\left|\frac{N-p}{p}\right|^p$ is optimal; see for example \cite{DH98} and \cite{HPL52}.

On the other hand, the so-called ``geometric Hardy inequality'' for $\Omega \subset \R^N$ being an open convex domain with smooth  boundary (possibly unbounded) reads as follows:
\begin{equation}\label{in.2}
\int_{\Omega}|\nabla u(x)|^p\,dx \geq \left(\frac{p-1}{p}\right)^p\int_{\Omega}\frac{|u(x)|^p}{\text{dist}(x,\partial \Omega)^p}dx\,,\quad p>1\,,
\end{equation}
for $u \in C_{0}^{\infty}(\Omega)$, where $\text{dist}(x,\partial \Omega)$ stands for the distance from $x \in \Omega$ to the boundary $\partial\Omega$, and the constant $\frac{1}{4}$ is optimal. Nowadays, there are many studies on the subject; c.f. the works \cite{BFT04},\cite{MS00}, \cite{MS97}, \cite{BFT18} in the Euclidean setting.

In Section \ref{Sec3}, Theorem \ref{thm.logHardyM}, we develop a unified strategy to prove a logarithmic Hardy inequality in an open set $M\subset \G$ of a stratified group $\G$, presuming the validity of the Hardy inequality with a suitable weight $w:M \rightarrow [0,\infty)$ in this setting. Logarithmic Hardy inequalities were introduced by Del Pino, Dolbeault, Filippas, and Tertikas in \cite{DDFT10}. 

The main ingredient to develop our method is the Sobolev inequality by Folland and Stein, see e.g. \cite{GV00}, that asserts that on any open set $\Omega\subset \mathbb{G}$, we have:
\begin{equation}
    \label{LS}
    \|u\|_{L^{p^{*}}(\Omega)}\leq A_p\left(\int_{\Omega}|\nabla_{H} u|^{p}dx\right)^{\frac{1}{p}},\; 1<p<Q,\; p^{*}=\frac{Qp}{Q-p}\,,\quad u\in C_{0}^{\infty}(\Omega)
\end{equation}
where $A_p>0$ depends on $p, \Omega$ and  $\G$, the vector-valued operator $\nabla_{H}$ stands for the horizontal gradient on $\G$, and $Q$ is its homogeneous dimension.  

Even though, there are several works on Hardy inequalities in the setting of stratified or even general Lie groups, c.f. \cite{CCR15}, \cite{L16}, \cite{RSS20}, and \cite{RS19} for a general overview of the topic, the study of their logarithmic counterpart has only been studied, to the best of our knowledge, only in  \cite{CKR21a}.

In Section \ref{Sec4}, Corollaries \ref{cor.Gplus} and \ref{cor.hor},  we employ our initial abstract approach and get two versions of the logarithmic Hardy inequality in the setting of a stratified group $\G$ of homogeneous dimension $Q \geq 3$, using the following two results that appeared in \cite[Theorem 8.1.3 (2)]{RS19} and \cite[Theorem 3.1]{RS17}, respectively.

\begin{thm}[Geometric Hardy inequality on $\G^{+}$]\label{thm.lh.rs1}
Let $\G^{+}$ be the half space on a stratified group $\G$. Then we have
\begin{equation}\label{hardy1}
\int_{\G^{+}}|\nabla_{H}u|^2\,dx \geq \frac{1}{4}\int_{\G^{+}}\frac{|u|^2}{\text{dist}(x,\partial\G^{+})^2}\,dx\,.
\end{equation}
\end{thm}

\begin{thm}[Horizontal Hardy inequality on $\G$]\label{thm.lh.rs2}
Let $\G$ be a stratified group of homogeneous dimension $Q \geq 3$ and topological dimension $n$. Let $N$ be the dimension of the first stratum of its Lie algebra, i.e., for $x \in \G$ we can write $x=(x',x'') \in \R^{N}\times \R^{n-N}$. Then we have  
\begin{equation}\label{hardy2}
\int_{\G}|\nabla_{H}u|^2\,dx \geq \left(\frac{N-2}{2}\right)^{2}\int_{\G}\frac{|u|^2}{|x'|^2}\,dx\,,
\end{equation}
where $|x'|$ is the Euclidean norm of $x'$ on $\R^{N}$.
\end{thm}

\subsection{Poincar\'e and Hardy-Poincar\'e inequalities}
Recall the classical probabilistic Poincar\'e inequality on $\R^n$:
\begin{equation}
    \label{Poi}
    \int_{\R^n}|f-\mu(f)|^{q}\,d\mu \leq  \mu(|\nabla f|^q)\,, \quad q \geq 1\,,
\end{equation}
where we have denoted $\mu(f):=\int_{\R^n}f\,d\mu$, and $d\mu$ stands for the Gaussian measure on $\R^n$. For the same measure, Beckner  in \cite{Bec89}, see also \cite{BD06}, proved the so-called ``generalised Poincar\'e inequality'' which reads as follows:  
\begin{equation}\label{Poi.bec}
\frac{1}{2-p}\left[\int_{\R^n}f^2\,d\mu-\left(\int_{\R^n}|f|^p\,d\mu \right)^{\frac{2}{p}} \right]\leq \int_{\R^n}|\nabla f|^2\,d\mu\,, \quad p \geq 1\,.
\end{equation}
To justify the terminology used to for the inequality \eqref{Poi.bec}, just observe that for such $\mu$ we have:
\begin{equation}\label{equiv.Poi}
\int_{\R^n}f^2\,d\mu-\left(\int_{\R^n}f\,d\mu\right)^2=\int_{\R^n}(f-\mu(f))^2\,d\mu\,.
\end{equation}
and so \eqref{Poi.bec} is equivalent when $p=1$ to the ``standard''   Poincar\'e inequality \eqref{Poi} for $q=2$. More generally, \eqref{equiv.Poi} is true on any measure space $\mathbb{X}$ when equipped with a probability measure $d\nu$ such that $\nu(\mathbb{X})=1$.

In Section \ref{Sec3}, Theorem \ref{thm.poi.M}, we develop a unified method to obtain a generalised (in the sense of \eqref{Poi.bec}) Hardy-Poincar\'e inequality on open set $M \subset \G$ of a stratified group $\G$, presuming the validity of the Hardy inequality with a suitable weight $w:M \rightarrow [0,\infty)$, and using the logarithmic version of the latter. Let us note that by ``generalised Hardy-Poincare'' inequality we regard an inequality of the form \eqref{Poi.bec} under the addition of the weight $w$ inherited by the corresponding Hardy inequality and in the setting of a stratified group $\G$. Let us mention that Hardy-Poincar\'e inequalities have also been considered in other settings, see for example \cite{DF22}, \cite{BBDGV07}.

In Section \ref{Sec4}, we apply the abstract theory developed earlier in Section \ref{Sec3}, and get several versions of the Hardy-Poincar\'e inequality  either on the half-space $\G^{+}$, or any open set $M \subset \G$, of a stratified group $\G$, when the regarded weight $w$ is the one appearing in \eqref{hardy1} or \eqref{hardy2}, respectively. In Theorem \ref{thm.gaus.Hardy.G} we show a version of the a ``horizontal Gross-type Hardy'' inequality with respect to the weight $w=w(x):=|x'|$ as in \eqref{hardy2}, and to a measure that is Gaussian only on the first stratum of the group $\G$. Consequently, and for the same measure, in Corollary \ref{cor.hor.poi}, we prove the ``horizontal'' version of the Hardy-Poincar\'e inequality in the stratified setting. Interestingly, as an application of the latter in the case where $\G=\R^n$, in Corollary \ref{cor1.poi}, we get a weighted version of the standard probabilistic Poincar\'e inequality \eqref{Poi} when $q=2$.

\section{Preliminaries}\label{Sec2}
An important class of nilpotent Lie algebras that allows for a dilation structure is that of graded Lie algebras. A graded Lie algebra $\mathfrak{g}$  allows for a vector space decomposition of the form $\mathfrak{g}=\oplus_{i=1}^{\infty} \mathfrak{g_i}$, where all, but finitely many $\mathfrak{g_i}$'s are $\{0\}$, and where $[\mathfrak{g_i},\mathfrak{g_j}]\subset \mathfrak{g_{i+j}}$ for all $i,j \in \mathbb{N}$. Recall that any connected, simply connected nilpotent Lie group with a graded Lie algebra is also called graded. Graded, or more generally homogeneous, Lie groups were initiated by Folland and Stein in \cite{FS82}; see also the monograph \cite{FR16} by Fischer and Ruzhansky for a detailed overview of the subject.

The so-called canonical dilation structure of a graded Lie algebra (and so also of a graded Lie group) arises from the matrix $A$ which is such that $AX'=j X'$, for $X' \in \mathfrak{g}_j$.  In this case the canonical dilations on $\mathfrak{g}$ are given by $D_r(X')=r^{j}X'$, for $X' \in \mathfrak{g_j}$, for $r>0$. The mapping $D_r$, for $r>0$, is an automorphism of $\mathfrak{g}$, and consequently also of the corresponding to $\mathfrak{g}$, graded Lie group, say $\G$. Hence if we write $x \in \G$ as
\[
x=(x^{(n_1)}, \ldots,x^{(n_{j_r})})\,, \quad \text{where} \quad x^{(n_j)}\in \R^{n_j}\,,
\]
then we have 
\[
D_r(x)=(rx^{(n_1)}, \ldots,r^{j_r}x^{(n_{j_r})})\,.
\]
For graded Lie groups, the quantity $n \leq Q:=\text{Tr}A$ is called the homogeneous dimension of $\G$. 

Additionally, if the vector fields $\{X_1,\cdots,X_{n_1}\}$ from $\mathfrak{g}_1$ generate, under the iterated Lie bracket operation $[\cdot,\cdot]: \mathfrak{g} \times \mathfrak{g} \rightarrow \mathfrak{g}$, the whole of $\mathfrak{g}$, then both $\mathfrak{g}$ and $\G$ are called stratified. In this case,  the dimension of the first stratum $\mathfrak{g}_1$ is $n_1$. For a stratified Lie group $\G$ the hypoelliptic vector valued operator $\nabla_{H}:=(X_1,\cdots,X_{n_1})$ is the extension of the gradient in this hypoelliptic setting, and is often called the horizontal gradient on $\G$.  

On a stratified Lie group $\G$, Garofalo in \cite{Gar08} defined the half-space of $\G$ by 
\[
\G^{+}:=\{ x\in \G : \langle x,v\rangle>d\}\,,
\]
where $d \in \R$, and the Riemannian outer unit normal to $\partial\G^{+}$ vector $v$ is given by $v=(v_1,\cdots,v_m)$, with $v_j \in \R^{n_j}$, for $j=1,\cdots,m$, where $m$ is the number of strata of  $\mathfrak{g}$. The Euclidean distance of an element $x \in \G$ to the boundary $\partial\G^{+}$, denoted by $\dd$ is defined as 
\[
\dd:=\langle x,v \rangle -d\,.
\]

\section{abstract results on the geometric Hardy and the geometric Hardy-Poincar\'e inequalities}\label{Sec3}
In this section we develop the general strategy to obtain a version of the geometric logarithmic Hardy and of the geometric Hardy-Poincar\'e inequality in the setting of a stratified group, assuming the we have the corresponding version of the Hardy inequality at our disposal.

\begin{thm}[Hardy Sobolev inequality]\label{thm.hardy.sM}
    For $\G$ being a stratified group of homogeneous dimension $Q \geq 3$, we let $M \subset \G$ be any open set in $\G$. Assume that on $M$ the following Hardy inequality holds true
    \begin{equation}\label{hardyM}
\int_{M}|\nabla_{H}u|^2\,dx \geq C_{H}\int_{M}\frac{|u|^2}{w(x)^2}\,dx\,,
\end{equation}
for some measurable, non-negative function $w$ on $M$ possible vanishing at a set of measure zero, and all $u \in C_{0}^{\infty}(M)$. Then  for $0\leq \beta \leq 2$, and for all $u \in C_{0}^{\infty}(M)$ we have
 \begin{equation}
        \label{hardy.sM}
         \int_{M}\frac{|u(x)|^{\frac{2(Q-\beta)}{Q-2}}}{\d^{\beta}}dx \leq  A_{2}^{\frac{Q(2-\beta)}{Q-2}} C_{H}^{-\frac{\beta}{2}} \left( \int_{M}|\nabla_{H}u|^2\,dx\right)^{\frac{(Q-\beta)}{Q-2}}\,,
    \end{equation}
    where $A_2$ is the constant in \eqref{LS} for $p=2$ and $\Omega=M$.
\end{thm}
\begin{proof}
   Using H\"{o}lder's inequality with $\frac{\beta}{2}+\frac{2-\beta}{2}=1$, we get
    \begin{equation*}
    \begin{split}
        \int_{M}\frac{|u(x)|^{\frac{2(Q-\beta)}{Q-2}}}{\d^{\beta}}dx&=\int_{M}\frac{|u(x)|^{\beta}|u(x)|^{\frac{2(Q-\beta)}{Q-2}-\beta}}{\d^{\beta}}dx\\&
\leq \left(\int_{M}\frac{|u(x)|^{2}}{\d^{2}}dx\right)^{\frac{\beta}{2}}\left(\int_{M}|u(x)|^{\frac{2Q}{Q-2}}\,dx\right)^{\frac{2-\beta}{2}}\\&
        \stackrel{\eqref{hardyM}}\leq  \left(C_{H}^{-1} \int_{M}|\nabla_{H}u|^2\,dx\right)^{\frac{\beta}{2}}
      \left(A_{2}^{2}\int_{M}|\nabla_{H}u|^2\,dx \right)^{\frac{Q(2-\beta)}{2(Q-2)}}\\&
      = A_{2}^{\frac{Q(2-\beta)}{Q-2}} C_{H}^{-\frac{\beta}{2}}
      \left(\int_{M}|\nabla_{H}u|^2\,dx \right)^{\frac{\beta}{2}+\frac{Q(2-\beta)}{2(Q-2)}}\,,
    \end{split}
    \end{equation*}
    where for the last inequality we have used the Sobolev inequality \eqref{LS} for $p=2$, and we have shown \eqref{hardy.sM} since $\frac{\beta}{2}+\frac{Q(2-\beta)}{2(Q-2)}=\frac{Q-\beta}{Q-2}$
    completing the proof.
\end{proof}
\begin{rem}
The Hardy-Sobolev inequality \eqref{hardy.sM} gives, for $\beta=0$, the Sobolev inequality \eqref{LS} on $M$, and for $\beta=2$, the geometric Hardy inequality \eqref{hardyM}.
\end{rem}
\begin{rem}\label{rem.Hn}
    In the case of the Heisenberg group $\mathbb{H}^n$ the constant $A_2$ as in \eqref{LS} is explicit, and in particular we have $A_2=\frac{(n!)^{\frac{1}{n+1}}}{\pi n^{2}}$. The latter is due to the optimal constant for the Sobolev inequality by Jerison and Lee \cite{JL88} in the aforesaid setting. In the Euclidean space $\R^n$ we have $A_2=(\pi n^2-2\pi n)^{-1/2}\frac{\Gamma(n)}{\Gamma(n/2)}$, see \cite{T76}, and $A_2$ is optimal.
\end{rem}
\begin{thm}[Geometric logarithmic Hardy inequality]\label{thm.logHardyM}
Let $\G$ be  a stratified group of homogeneous dimension $Q\geq  3$, and let $M \subset \G$ be an open set in $\G$. Assume also that \eqref{hardyM} holds true for some non-negative measurable function $w$ on $M$. Then, for $0\leq \beta < 2$, and for all $u \in C_{0}^{\infty}(M)$ we have
    
     \begin{equation}\label{log.hardyM}
\begin{split}
    &\int_{M}\frac{\frac{|u|^{2}}{\d^{2-\beta}}}{\left\|\frac{u}{w^{\frac{2-\beta}{2}}}\right\|_{L^2(M)}^2}\log\left(\frac{|u|^{2}\d^{2Q-4}}{\left\|\frac{u}{w^{\frac{2-\beta}{2}}}\right\|_{L^2(M)}^2}\right)\,dx \\ &\leq\frac{Q-\beta}{2-\beta}\log\left(\frac{C_{LH}\int_{M}|\nabla_{H}u|^2\,dx }{\left\|\frac{u}{w^{\frac{2-\beta}{2}}}\right\|_{L^2(M)}^2}\right)\,,
\end{split}
\end{equation}
where  
 \begin{equation}\label{CQbM}
   C_{LH}:=A_{2}^{\frac{2Q(2-\beta)}{Q-\beta}}C_{H}^{-\frac{\beta(Q-2)}{Q-\beta}}\,.
   \end{equation}
 In particular, when $\left\|\frac{u}{w^{\frac{2-\beta}{2}}}\right\|_{L^2(M)}=1$, we have 
 
\begin{equation}\label{L1=1}
        \int_{M}\frac{|u|^{2}}{\d^{2-\beta}}\log \left(|u(x)| \d^{Q-2} \right)\,dx
        \leq \frac{Q-\beta}{2(2-\beta)} \log \left(C_{LH}\|\nabla_{H}u\|^{2}_{L^2(M)}\right)\,.
\end{equation}
\end{thm}

\begin{proof}
Let us introduce the auxiliary parameters $\omega^{*}:=\frac{2(Q-\beta)}{Q-2}$ and $q \in (2, \omega^{*})$. For some fixed $q$, we set $\omega=\frac{2(Q-\beta)}{2-\beta}-\frac{q(Q-2)}{2-\beta}$, and notice that 
\begin{equation}\label{eq.omega}
    \beta \frac{2-\omega}{2}-\beta=-\frac{\omega \beta}{2}\,,
\end{equation}
and 
\begin{equation}
    \label{eq.omega2}
    Q-\beta-\frac{q}{2}(Q-2)-\omega=-\frac{\omega \beta}{2}\,.
\end{equation}
 Hence by H\"{o}lder's inequality for $\frac{\omega}{2}+\frac{2-\omega}{2}=1$ we have 
\begin{equation*}
    \begin{split}
        \int_{M}\frac{|u|^{q}}{\d^{Q-\beta-\frac{q}{2}(Q-2)}}\,dx &=\int_{M}\frac{|u|^{\omega}}{\d^{\omega }}\frac{|u|^{q-\omega}}{\d^{Q-\beta-\frac{q}{2}(Q-2)-\omega }}\,dx\\&
       \stackrel{\eqref{eq.omega2}} =\int_{M}\frac{|u|^{\omega}}{\d^{\omega }}\frac{|u|^{q-\omega}}{\d^{-\frac{\omega \beta}{2}}}\,dx\\&
        \stackrel{\eqref{eq.omega}}  =\int_{M}\frac{|u|^{\omega}}{\d^{\omega-\beta}}\frac{|u|^{q-\omega}}{\d^{\beta\frac{2-\omega}{2}}}d\lambda(x)\\&
        \leq \left(\int_{M}\frac{|u|^{2}}{\d^{2-\frac{2\beta }{\omega}}}\,dx\right)^{\frac{\omega}{2}}\left(\int_{M}\frac{|u|^{\frac{2(q-\omega)}{2-\omega}}}{\d^{\beta}}\,dx\right)^{\frac{2-\omega}{2}}\,.
    \end{split}
\end{equation*}
Now, under  the observation that 
\[
\frac{\omega}{2}=\frac{\omega^{*}-q}{\omega^{*}-2} \quad \text{and} \quad \frac{2(q-\omega)}{2-\omega}=\omega^{*}\,,
\]
the above inequalities can be summarised as
\begin{multline}\label{omega*}
        \int_{M}\frac{|u|^{q}}{\d^{Q-\beta-\frac{q}{2}(Q-2)}}\,dx
         \leq \left(\int_{M}\frac{|u|^{2}}{\d^{2-\beta\frac{\omega^{*}-2}{\omega^{*}-q}}}\,dx\right)^{\frac{\omega^{*}-q}{\omega^{*}-2}} \\
       \times  \left(\int_{M}\frac{|u|^{\omega^{*}}}{\d^{\beta}}\,dx\right)^{\frac{q-2}{\omega^{*}-2}}\,.
\end{multline}
The inequality \eqref{omega*} for $q=2+\delta$, can be rewritten as 
\begin{multline}\label{fordelta}
        \int_{M}\frac{|u|^{2+\delta}}{\d^{Q-\beta-\frac{2+\delta}{2}(Q-2)}}\,dx
         \leq \left(\int_{M}\frac{|u|^{2}}{\d^{2-\beta\frac{\omega^{*}-2}{\omega^{*}-(2+\delta)}}}\,dx\right)^{\frac{\omega^{*}-(2+\delta)}{\omega^{*}-2}} \\
         \left(\int_{M}\frac{|u|^{\omega^{*}} }{\d^{\beta}}\,dx\right)^{\frac{(2+\delta)-2}{\omega^{*}-2}}\,.
\end{multline}
Taking the limit as $\delta \rightarrow 0$ of \eqref{fordelta} we get 
\begin{multline}\label{delta=0}
         \int_{M}\frac{|u|^{2}}{\d^{Q-\beta-\frac{2}{2}(Q-2)}}\,dx
         \leq \left(\int_{M}\frac{|u|^{2}}{\d^{2-\beta\frac{\omega^{*}-2}{\omega^{*}-2}}}\,dx\right)^{\frac{\omega^{*}-2}{\omega^{*}-2}} \\
         \left(\int_{M}\frac{|u|^{\omega^{*}}}{\d^{\beta}}\,dx\right)^{\frac{2-2}{\omega^{*}-2}}\,.
\end{multline}
Observe that the second term of the right hand side of \eqref{delta=0} equals to $1$, while  the left hand side and the first term of the right hand side of \eqref{delta=0} are equal. The latter implies that \eqref{delta=0} holds true as an equality, i.e., we have  
\begin{multline}\label{delta.eq}
         \int_{M}\frac{|u|^{2}}{\d^{Q-\beta-\frac{2}{2}(Q-2)}}\,dx
         =\left(\int_{M}\frac{|u|^{2}}{\d^{2-\beta\frac{\omega^{*}-2}{\omega^{*}-2}}}\,dx\right)^{\frac{\omega^{*}-2}{\omega^{*}-2}} \\
         \left(\int_{M}\frac{|u|^{\omega^{*}}}{\d^{\beta}}\,dx\right)^{\frac{2-2}{\omega^{*}-2}}\,.
\end{multline}
Using \eqref{delta.eq} and \eqref{fordelta}, we get 
\begin{equation}\label{difference}
\begin{split}
        &\lim_{\delta\rightarrow 0}\frac{1}{\delta}  \int_{M} \left(\frac{|u|^{2+\delta}}{\d^{Q-\beta-\frac{2+\delta}{2}(Q-2)}}- \frac{|u|^{2}}{\d^{Q-\beta-\frac{2}{2}(Q-2)}}\right)\,dx \\&
        \leq \lim_{\delta\rightarrow 0}\frac{1}{\delta}\Biggl{[}\left(\int_{M}\frac{|u|^{2}}{\d^{2-\beta\frac{\omega^{*}-2}{\omega^{*}-(2+\delta)}}}\,dx\right)^{\frac{\omega^{*}-(2+\delta)}{\omega^{*}-2}} 
         \left(\int_{M}\frac{|u|^{\omega^{*}}}{\d^{\beta}}\,dx\right)^{\frac{(2+\delta)-2}{\omega^{*}-2}}\\&
        -\left(\int_{M}\frac{|u|^{2}}{\d^{2-\beta\frac{\omega^{*}-2}{\omega^{*}-2}}}\,dx\right)^{\frac{\omega^{*}-2}{\omega^{*}-2}} \left(\int_{M}\frac{|u|^{\omega^{*}}}{\d^{\beta}}\,dx\right)^{\frac{2-2}{\omega^{*}-2}}\Biggl{]}:=\lim_{\delta\rightarrow 0}\frac{1}{\delta} I(\delta).
\end{split}
\end{equation}
Using the l'H\^{o}pital rule in the variable $\delta$, we compute the left-hand side of the inequality \eqref{difference}:
\begin{equation}\label{delH1}
    \begin{split}
       & \lim_{\delta\rightarrow 0}\frac{1}{\delta} \left(\frac{|u|^{2+\delta}}{\d^{Q-\beta-\frac{2+\delta}{2}(Q-2)}}- \frac{|u|^{2}}{\d^{Q-\beta-\frac{2}{2}(Q-2)}}\right)\\&=\lim_{\delta\rightarrow 0}\frac{d}{d\delta}\frac{|u|^{\delta+2}}{\d^{Q-\beta-\frac{(\delta+2)}{2}(Q-2)}}\\&
        =\frac{|u|^{2}\log(|u|^{2}\d^{Q-2})}{2\d^{2-\beta}}\,.
    \end{split}
\end{equation}
Using the  l'H\^{o}pital rule we will also compute the derivative of the right hand side of \eqref{difference}. To this end, let us first compute the derivative 
\[
\frac{d}{d\delta} \left(\int_{M}\frac{|u|^{2}}{\d^{2-\beta\frac{\omega^{*}-2}{\omega^{*}-(2+\delta)}}}\,dx\right)^{\frac{\omega^{*}-(2+\delta)}{\omega^{*}-2}} := \frac{d}{d\delta} (f(\delta)^{g(\delta)})\,,
\]
where we have defined 
\[
f(\delta):=\int_{M}\frac{|u|^{2}}{\d^{2-\beta\frac{\omega^{*}-2}{\omega^{*}-(2+\delta)}}}\,dx\,,
\]
and 
\[
g(\delta):= \frac{\omega^{*}-(2+\delta)}{\omega^{*}-2}\,.
\]
Since 
\begin{equation}\label{df}
    \begin{split}
       \frac{df(\delta)}{d\delta}&=\frac{d}{d\delta}\int_{M}\frac{|u|^{2}}{\d^{2-\beta\frac{\omega^{*}-2}{\omega^{*}-(\delta+2)}}}\,dx\\&
       =\frac{(\beta-2)(\omega^{*}-2)}{(\omega^{*}-(\delta+2))^{2}}\int_{M}\frac{|u|^{2}}{\d^{2-\beta\frac{\omega^{*}-2}{\omega^{*}-(\delta+2)}}}\log(\d)\,dx\,,
    \end{split}
\end{equation}
and 
\begin{equation}\label{dg}
    \frac{dg(\delta)}{d\delta}=\frac{d}{d\delta}\frac{\omega^{*}-(\delta+2)}{\omega^{*}-2}=-\frac{1}{\omega^{*}-2}\,,
\end{equation}
we have
\begin{equation}\label{dfg}
\begin{split}
    \frac{d}{d\delta} (f(\delta)^{g(\delta)})&=(f(\delta)^{g(\delta)}\left(\frac{dg(\delta)}{d\delta}\log(f(\delta))+\frac{g(\delta)\frac{df(\delta)}{d\delta}}{f(\delta)}\right)\\&
    \stackrel{\eqref{df},\eqref{dg}}=(f(\delta))^{g(\delta)}\Biggl{(}-\frac{1}{\omega^{*}-2}\log\left(\int_{M}\frac{|u|^{2}}{\d^{2-\beta\frac{\omega^{*}-2}{\omega^{*}-(\delta+2)}}}\,dx\right)\\&
    +\frac{\beta-2}{\omega^{*}-(\delta+2)}\frac{\int_{M}\frac{|u|^{2}}{\d^{2-\beta\frac{\omega^{*}-2}{\omega^{*}-(\delta+2)}}}\log(\d)\,dx}{\int_{M}\frac{|u|^{2}}{\d^{2-\beta\frac{\omega^{*}-2}{\omega^{*}-(\delta+2)}}}\,dx}\Biggl{)}.
\end{split}
\end{equation}
The following derivative will also be used in the computations that follow
\begin{equation}\label{der2}
    \begin{split}
      \frac{d}{d\delta}\left(\int_{M}\frac{|u|^{\omega^{*}}}{\d^{\beta}}\,dx\right)^{\frac{(\delta+2)-2}{\omega^{*}-2}}&= \frac{1}{\omega^{*}-2}\left(\int_{M}\frac{|u|^{\omega^{*}}}{\d^{\beta}}\,dx\right)^{\frac{(\delta+2)-2}{\omega^{*}-2}}\\ &
      \times \log\left(\int_{M}\frac{|u|^{\omega^{*}}}{\d^{\beta}}\,dx\right)\,.
    \end{split}
\end{equation}
Now, for the right hand side of \eqref{difference} we have 
\begin{equation}\label{I}
    \begin{split}
        &\lim_{\delta\rightarrow 0}\frac{1}{\delta} I(\delta)
        =\lim_{\delta\rightarrow 0}\frac{d}{d\delta}\left[\left(\int_{M}\frac{|u|^{2}}{\d^{2-\beta\frac{\omega^{*}-2}{\omega^{*}-(\delta+2)}}}\,dx\right)^{\frac{\omega^{*}-(\delta+2)}{\omega^{*}-2}}\left(\int_{M}\frac{|u|^{\omega^{*}}}{\d^{\beta}}\,dx\right)^{\frac{(\delta+2)-2}{\omega^{*}-2}}\right]\\&
        =\lim_{\delta\rightarrow 0}\frac{d}{d\delta}\left((f(\delta))^{g(\delta)}\left(\int_{M}\frac{|u|^{\omega^{*}}}{\d^{\beta}}\,dx\right)^{\frac{(\delta+2)-2}{\omega^{*}-2}}\right)\\&
        =\lim_{\delta\rightarrow 0}\left[\frac{d(f(\delta))^{g(\delta)}}{d\delta}\left(\int_{M}\frac{|u|^{\omega^{*}}}{\d^{\beta}}\,dx\right)^{\frac{(\delta+2)-2}{\omega^{*}-2}}+(f(\delta))^{g(\delta)}\frac{d}{d\delta}\left(\int_{M}\frac{|u|^{\omega^{*}}}{\d^{\beta}}\,dx\right)^{\frac{(\delta+2)-2}{\omega^{*}-2}}\right]\\&
    \stackrel{\eqref{dfg},\eqref{der2}}=\frac{\int_{M}\frac{|u|^{2}}{\d^{2-\beta}}\,dx}{\omega^{*}-2}\Biggl{[}-\log\left(\int_{M}\frac{|u|^{2}}{\d^{2-\beta}}\,dx\right)\\&
    +(\beta-2)\frac{\int_{M}\frac{|u|^{2}}{\d^{2-\beta}}\log(\d)\,dx}{\int_{M}\frac{|u|^{2}}{\d^{2-\beta}}\,dx}+\log\left(\int_{M}\frac{|u|^{\omega^{*}}}{\d^{\beta}}\,dx\right)\Biggl{]}\\&
    =\frac{\int_{M}\frac{|u|^{2}}{\d^{2-\beta}}\,dx}{\omega^{*}-2}\Biggl{[}(\beta-2)\frac{\int_{M}\frac{|u|^{2}}{\d^{2-\beta}}\log(\d)\,dx}{\int_{M}\frac{|u|^{2}}{\d^{2-\beta}}\,dx}+\log\left(\frac{\int_{M}\frac{|u|^{\omega^{*}}}{\d^{\beta}}\,dx}{\int_{M}\frac{|u|^{2}}{\d^{2-\beta}}\,dx}\right)\Biggl{]}\\&
    =\frac{I_{1}}{\omega^{*}-2}\left((\beta-2)\frac{I_{3}}{I_{1}}+\log\frac{I_{2}}{I_{1}}\right)\,,
\end{split}
\end{equation}
where we have set 
\[
    I_{1}=\int_{M}\frac{|u|^{2}}{\d^{2-\beta}}\,dx,
\]
\[
    I_{2}=\int_{M}\frac{|u|^{\omega^{*}}}{\d^{\beta}}\,dx,
\]
and 
\[
    I_{3}=\int_{M}\frac{|u|^{2}}{\d^{2-\beta}}\log(\d)\,dx.
\]
Now, in view of expressing the quantity 
\[
\frac{I_{1}}{\omega^{*}-2}\left((\beta-2)\frac{I_{3}}{I_{1}}+\log\frac{I_{2}}{I_{1}}\right)
\]
in a way that is more convenient for our computations, we have 
\begin{eqnarray}\label{I123}
    \frac{I_{1}}{\omega^{*}-2}\left((\beta-2)\frac{I_{3}}{I_{1}}+\log\frac{I_{2}}{I_{1}}\right) & = & \frac{\beta-2}{\omega^{*}-2}I_{3}+\frac{I_{1}}{\omega^{*}-2}\frac{2\omega^{*}}{2\omega^{*} }\log\frac{I_{2}}{I_{1}}\nonumber\\
    & = & \frac{\beta-2}{\omega^{*}-2}I_{3}+\frac{I_{1}\omega^{*}}{(\omega^{*}-2)2}\log\frac{I^{\frac{2}{\omega^{*} }}_{2}}{I_{1}^{1-1+\frac{2}{\omega^{*}}}}\nonumber\\
    & = & \frac{\beta-2}{\omega^{*}-2}I_{3}+\frac{I_{1}\omega^{*}}{(\omega^{*}-2)2}\log\frac{I^{\frac{2}{\omega^{*} }}_{2}}{I_{1}}-\frac{I_{1}\omega^{*} }{(\omega^{*}-2)2}\log I_{1}^{-1+\frac{2}{\omega^{*}}}\nonumber\\
     & = & \frac{\beta-2}{\omega^{*}-2}I_{3}+\frac{I_{1}\omega^{*}}{(\omega^{*}-2)2}\log\frac{I^{\frac{2}{\omega^{*} }}_{2}}{I_{1}}+\frac{I_{1}}{2}\log I_{1}\,.
\end{eqnarray}
A combination of \eqref{I} together with \eqref{I123} gives 
\begin{equation}\label{II123}
  \lim_{\delta\rightarrow 0}\frac{1}{\delta} I(\delta)=\frac{\beta-2}{\omega^{*}-2}I_{3}+\frac{I_{1}\omega^{*}}{(\omega^{*}-2)2}\log\frac{I^{\frac{2}{\omega^{*} }}_{2}}{I_{1}}+\frac{I_{1}}{2}\log I_{1}\,. 
\end{equation}

Plugging \eqref{II123} and \eqref{delH1} into \eqref{difference} yields

\begin{equation}\label{2last.est}
    \begin{split}
       & \int_{M} \frac{|u|^{2}\log(|u|^{2}\d^{Q-2})}{\d^{2-\beta}}\,dx\\&\leq  \frac{2(\beta-2)}{\omega^{*}-2}I_{3}+\frac{I_{1}\omega^{*}}{(\omega^{*}-2)}\log\frac{I^{\frac{2}{\omega^{*} }}_{2}}{I_{1}}+I_{1}\log I_{1}\\&
        = \int_{M}\frac{|u|^{2}}{\d^{2-\beta}}\log\left(\d^{\frac{(\beta-2)(Q-2)}{2-\beta}}\int_{M}\frac{|u|^{2}}{\d^{2-\beta}}\,dx\right)dx\\&
        + \frac{(Q-\beta)\int_{M}\frac{|u|^{2}}{\d^{2-\beta}}\,dx}{2-\beta}\log\left[\frac{\left(\int_{M}\frac{|u|^{\omega^{*}}}{\d^{\beta}}\,dx\right)^{\frac{2}{\omega^{*}}}}{\int_{M}\frac{|u|^{2}}{\d^{2-\beta}}\,dx}\right]\,,
    \end{split}
\end{equation}
where we have substituted the expressions for $I_1, I_2$ and $I_3$ and have used the fact that 
\[
\frac{\omega^{*}}{\omega^{*}-2}=\frac{Q-\beta}{2-\beta}\quad \text{and}\quad \frac{2(\beta-2)}{\omega^{*}-2}=\frac{(\beta-2)(Q-2)}{Q-\beta}\,.
\] 
Now, rearranging the terms of \eqref{2last.est} using the properties of the logarithm we get 
\begin{equation}\label{last.est}
\begin{split}
    &\int_{M}\frac{\frac{|u|^{2}}{\d^{2-\beta}}}{\int_{M}\frac{|u|^{2}}{\d^{2-\beta}}\,dx}\log\left(\frac{|u|^{2}\d^{2Q-4}}{\int_{M}\frac{|u|^{2}}{\d^{2-\beta}}\,dx}\right)\,dx \\ &\leq\frac{Q-\beta}{2-\beta}\log\left[\frac{\left(\int_{M}\frac{|u|^{\omega^{*}}}{\d^{\beta}}\,dx\right)^{\frac{2}{\omega^{*}}}}{\int_{M}\frac{|u|^{2}}{\d^{2-\beta}}\,dx}\right]\,.
\end{split}
\end{equation}
Now, by Theorem \ref{thm.hardy.sM},  since $\omega^{*}=\frac{2(Q-\beta)}{Q-2}$, we have
\begin{eqnarray}
    \label{appl.thm.hardy.s}
    \left(\int_{M}\frac{|u|^{\omega^{*}}}{\d^{\beta}}\,dx\right)^{\frac{2}{\omega^{*}}}& \leq & A_{2}^{\frac{Q(2-\beta)}{Q-2} \frac{2}{\omega^{*}}} C_{H}^{-\frac{\beta}{\omega^{*}}} \left( \int_{M}|\nabla_{H}u|^{2}\,dx\right)^{\frac{2}{\omega^{*}}\left(\frac{Q-\beta}{Q-2} \right)}\,.
\end{eqnarray}
Now, since 
   \[
   \frac{2}{\omega^{*}}\left(\frac{Q-\beta}{Q-2} \right)=2 \frac{Q-2}{2(Q-\beta)}\frac{Q-\beta}{Q-2}=1\,,
   \] 
   and 
   \[
   \frac{Q(2-\beta)}{Q-2} \frac{2}{\omega^{*}}=\frac{Q(2-\beta)}{Q-2} 2 \frac{Q-2}{2(Q-\beta)}=\frac{2Q(2-\beta)}{Q-\beta}\,,
   \]
   inequality \eqref{appl.thm.hardy.s} can be written as 
   \begin{equation}
       \label{appl.hardy.2}
        \left(\int_{M}\frac{|u|^{\omega^{*}}}{\d^{\beta}}\,dx\right)^{\frac{2}{\omega^{*}}} \leq C_{LH} \int_{M}|\nabla_{H}u|^{2}\,dx\,,
   \end{equation}
   where 
   \[
   C_{LH}:=A_{2}^{\frac{2Q(2-\beta)}{Q-\beta}}C_{H}^{-\frac{\beta(Q-2)}{Q-\beta}}\,.
   \]
   Finally, a combination of \eqref{appl.hardy.2} and \eqref{last.est} gives 
   \begin{equation*}
\begin{split}
    &\int_{M}\frac{\frac{|u|^{2}}{\d^{2-\beta}}}{\int_{M}\frac{|u|^{2}}{\d^{2-\beta}}\,dx}\log\left(\frac{|u|^{2}\d^{2Q-4}}{\int_{M}\frac{|u|^{2}}{\d^{2-\beta}}\,dx}\right)\,dx \\&
    \leq\frac{Q-\beta}{2-\beta}\log\left(\frac{C_{LH}\int_{M}|\nabla_{H}u|^2\,dx}{\int_{M}\frac{|u|^{2}}{\d^{2-\beta}}\,dx}\right)\,,
\end{split}
\end{equation*}
and we have completed the proof.
\end{proof}
\begin{rem}
    In view of Remark \ref{rem.Hn}, the constant $C_{LH}$ in the logarithmic Hardy inequality as in \eqref{log.hardyM} in the case of the Heisenberg group $\mathbb{H}^n$ is explicit and depends only on $\beta$. In particular, using \eqref{appl.thm.hardy.s} for $A_2=\frac{(n!)^{\frac{1}{n+1}}}{\pi n^2}$, we can calculate
    \[
    C_{LH}=\left(\frac{(n!)^{\frac{1}{n+1}}}{\pi n^2} \right)^{\frac{2Q(2-\beta)}{Q-\beta}}C_{H}^{-\frac{\beta(Q-2)}{Q-\beta}}\,,
    \]
    where $Q=2n+2$. Similarly in the case of $\R^n$, choosing $A_2=(\pi n^2-2\pi n)^{-1/2}\frac{\Gamma(n)}{\Gamma(n/2)}$ we have
    \[
     C_{LH}=\left((\pi n^2-2\pi n)^{-1/2}\frac{\Gamma(n)}{\Gamma(n/2)}\right)^{\frac{2n(2-\beta)}{n-\beta}}C_{H}^{-\frac{\beta(n-2)}{n-\beta}}\,.
    \] 
\end{rem}
Choosing $\beta=0$ in \eqref{log.hardyM} we get $C_{LH}=A_{2}^{4}$, and Theorem \ref{thm.logHardyM} implies the following:
\begin{cor}
Under the hypothesis of Theorem \ref{thm.logHardyM}, for all $u \in C_{0}^{\infty}(M)$, we have
\[
\int_{M}\frac{\frac{|u|^{2}}{\d^{2}}}{\left\|\frac{u}{w}\right\|_{L^2(M)}^2}\log\left(\frac{|u|^{2}\d^{2Q-4}}{\left\|\frac{u}{w}\right\|_{L^2(M)}^2}\right)\,dx \leq\frac{Q}{2}\log\left(\frac{A_{2}^{4}\int_{M}|\nabla_{H}u|^2\,dx}{\left\|\frac{u}{w}\right\|_{L^2(M)}^2}\right)\,.
\]  
\end{cor}

The following inequality that we call ``geometric Hardy-Poincar\'e inequality'' follows by the geometric logarithmic Hardy inequality \eqref{log.hardyM} and can be viewed as a generalisation on stratified groups of the generalised Poincar\'e inequality \cite{Bec89} with weights. 
\begin{thm}[Geometric Hardy-Poincar\'e inequality]\label{thm.poi.M}
    Let $\G$ be a stratified group of homogeneous dimension $Q \geq 3$, and let $M \subset \G$ be an open subset of $\G$. Assume also that \eqref{hardyM} holds true for some measurable non-negative function $w$ on $M$. Then, for $0\leq \beta < 2$, $q>1$, and for all $u \in C_{0}^{\infty}(M)$ we have
   \begin{equation}\label{p.hardy.M}
\begin{split}
     &\int_{M} \frac{\frac{|u|^2}{\d^{2-\beta}}}{\left\|\frac{u}{w^{\frac{2-\beta}{2}}}\right\|_{L^2(M)}^2}\,dx-\left(\int_{M}\frac{\frac{|u|^{\frac{2}{q}}}{\d^{\frac{2-\beta}{q}}}}{\left\|\frac{u}{w^{\frac{2-\beta}{2}}}\right\|_{L^2(M)}^{\frac{2}{q}}}\,dx \right)^{q} \\&
      \leq  \frac{(q-1)(Q-\beta)}{2-\beta} \log\left(\frac{C_{LH}\int_{M}|\nabla_{H}u|^2\,dx}{\left\|\frac{u}{w
^{\frac{2-\beta}{2}}}\right\|_{L^2(M)}^2}\right)+J(u)\,,
      \end{split}
\end{equation}

where $C_{LH}$ is as in \eqref{CQbM} and  we have defined $J(u)$ by 
\[
J(u):= (q-1) \int_{M}\frac{\frac{|u|^2}{\d^{2-\beta}}}{\left\|\frac{u}{w^{\frac{2-\beta}{2}}}\right\|^{2}_{L^2(M)}}\log \left(\frac{|u|^2 \d^{(\beta-2)-(2Q-4)}}{\left\|\frac{u}{w^{\frac{2-\beta}{2}}}\right\|^{2}_{L^2(M)}}\right)\,dx\,.
\]  
\end{thm}
\begin{proof}
     For $g$ such that $\|g\|_{L^2(\G)}=1$ we consider the function $b=b(q):=q \log \left( \int_{M}|g|^{\frac{2}{q}}\,dx \right)$ for $q> 1$. If we set
   \[
   I(q):=\int_{M} |g|^{\frac{2}{q}}\,dx\,,
   \]
   then we have
   \begin{equation*}
       b(q)=q\log I(q).
   \end{equation*}
   Now, since
   \begin{equation*}
       I'(q)=\int_{M}\log |g||g|^{\frac{2}{q}}\left(-\frac{2}{q^2}\right)\,dx,
   \end{equation*}
   
    \begin{equation*}
       I''(q)=\int_{M}\left((\log|g|)^{2}|g|^{\frac{2}{q}}\left(-\frac{2}{q^2}\right)^{2}+\log |g||g|^{\frac{2}{q}}\left(\frac{4}{q^3}\right)\right)\,dx\,,
   \end{equation*}
  and
   \[
   b'(q)=\log I(q)+\frac{q I'(q)}{I(q)}\,,
   \]
   we get
   \begin{equation}
   \begin{split}
       b''(q)&=\frac{I(q)(2I'(q)+qI''(q))-q(I'(q))^{2}}{I^{2}(q)}\\&
       =\frac{q}{I^{2}(q)}\Biggl{[}\left(\int_{M}|g|^{\frac{2}{q}}d\mu\right)\left(\int_{M}(\log|g|)^{2}|g|^{\frac{2}{q}}\left(-\frac{2}{q^2}\right)^{2}dx\right)\\&
       -\left(\int_{M}\log |g||g|^{\frac{2}{q}}\left(-\frac{2}{q^2}\right)\,dx\right)^{2}\Biggl{]}.
   \end{split}
   \end{equation}
   Observe  that $b$ is a convex function, i.e., we have $b'' \geq 0$, since by the Cauchy-Schwarz inequality we have 
   \[
   \left( \int_{M}|g|^{\frac{2}{q}}\,dx\right)\left(\int_{M}(\log |g|)^2 |g|^{\frac{2}{q}}\left(-\frac{2}{q^2}\right)^2\,dx \right) \geq \left(\int_{M}\log |g||g|^{\frac{2}{q}}\left(-\frac{2}{q^2}\right)\,dx \right)^2\,.
   \]
 Now, the convexity of $b$ implies  the convexity of the function $q \mapsto e^{b(q)}$, and the latter implies that the function 
   \[
   \phi(q):= \frac{e^{b(1)}-e^{b(q)}}{1-q}\,,
   \]
   is monotonically non-increasing for $q \geq 1$ \footnote{Recall that a function $f$ on $\mathbb{R}$ is convex if and only if the function $Q(x,y)=\frac{f(y)-f(x)}{y-x}$, $x,y \in \mathbb{R}$, is monotonically non-decreasing in $x$ for fixed $y$.}. Consequently, we get 
   \begin{equation}\label{phi.decreasing}
   \phi(q)\leq \lim_{q \rightarrow 1}\phi(q)\,,
   \end{equation}
  where
   \begin{equation}\label{phi(q)}
   \phi(q)=\frac{\int_{M}|g|^2\,dx-\left( \int_{M}|g|^{\frac{2}{q}}\,dx\right)^q}{1-q}\,.
   \end{equation}
   For the computation of the limit $\lim\limits_{q \rightarrow 1}\phi(q)$ we will use the l'H\^{o}pital rule.
   We have 
   \begin{eqnarray*}
       \frac{d}{dq}\left(\int_{M}|g|^{\frac{2}{q}}\,dx \right)^q & = &
       \left(\int_{M}|g|^{\frac{2}{q}}\,dx \right)^q \frac{d}{dq} \left[q \log \left(\int_{M}|g|^{2/q}\,dx \right) \right] \\
      & = &  \left(\int_{M}|g|^{\frac{2}{q}}\,dx \right)^q \left[\log \left(\int_{M}|g|^{2/q}\,dx \right)-\frac{2}{q}\frac{\int_{M}\log |g||g|^{2/q}dx}{\int_{M}|g|^{2/q}dx}  \right],
   \end{eqnarray*}
   which implies 
   \begin{eqnarray*}
     \lim_{q \rightarrow 1}\frac{d}{dq}\left(\int_{M}|g|^{\frac{2}{q}}\,dx \right)^q & = & \int_{M}|g|^2dx \cdot \log \left(\int_{M}|g|^2dx \right)-2\int_{M}\ |g|^2 \log|g|\,dx\\
     & = & -2\int_{M}\ |g|^2 \log|g|\,dx\,,
   \end{eqnarray*}
   since $\|g\|_{L^2(\G)}=1$. So, we have
    $\lim\limits_{q \rightarrow 1}\phi(q)=2 \int_{M}\ |g|^2 \log|g|\,dx$. Combining this together with \eqref{phi.decreasing} and \eqref{phi(q)} we get 
   \begin{equation*}
       \int_{M}|g|^2\,dx-\left( \int_{M}|g|^{\frac{2}{q}}\,dx\right)^q \leq 2(q-1) \int_{M}|g|^2 \log |g|\,dx\,.
   \end{equation*}
   If we replace $g(x)$ by $\frac{u(x)}{\d^{\frac{2-\beta}{2}}}\left\|\frac{u}{w^{\frac{2-\beta}{2}}}\right\|_{L^2(M)}^{-1}$, the latter inequality reads as follows
   \begin{equation}\label{diff}
       \begin{split}
           &\int_{M}\frac{\frac{|u|^2}{\d^{2-\beta}}}{\left\|\frac{u}{w^{\frac{2-\beta}{2}}}\right\|_{L^2(M)}^2}\,dx-\left(\int_{M}\frac{\frac{|u|^{\frac{2}{q}}}{\d^{\frac{2-\beta}{q}}}}{\left\|\frac{u}{w^{\frac{2-\beta}{2}}}\right\|_{L^2(M)}^{\frac{2}{q}}}\,dx \right)^{q} \\&
           \leq(q-1) \int_{M}\frac{\frac{|u|^2}{\d^{2-\beta}}}{\left\|\frac{u}{w^{\frac{2-\beta}{2}}}\right\|_{L^2(M)}^2}\log \left(\frac{|u|^2\d^{\beta-2}}{\left\|\frac{u}{w^{\frac{2-\beta}{2}}}\right\|_{L^2(M)}^2}\right)\,dx \\&
           = (q-1) \int_{M}\frac{\frac{|u|^2}{\d^{2-\beta}}}{\left\|\frac{u}{w^{\frac{2-\beta}{2}}}\right\|^{2}_{L^2(M)}}\log \left(\frac{|u|^2 \d^{2Q-4}}{\left\|\frac{u}{w^{\frac{2-\beta}{2}}}\right\|^{2}_{L^2(M)}}\right)\,dx\\&
           +(q-1) \int_{M}\frac{\frac{|u|^2}{\d^{2-\beta}}}{\left\|\frac{u}{w^{\frac{2-\beta}{2}}}\right\|^{2}_{L^2(M)}}\log \left(\frac{|u|^2 \d^{(\beta-2)-\left(2Q-4 \right)}}{\left\|\frac{u}{w^{\frac{2-\beta}{2}}}\right\|^{2}_{L^2(M)}}\right)\,dx
       \end{split}
   \end{equation}
  where for the last equality we have used the properties of the logarithm and the equality 
  \[
  \beta-2=2Q-4+\left[\beta-2-\left(2Q-4\right) \right]\,.
  \]
Hence a combination of \eqref{diff} together with the the logarithmic Hardy inequality \eqref{log.hardyM} gives
\begin{equation}\label{poin.hardy.I}
\begin{split}
      &\int_{M}\frac{\frac{|u|^2}{\d^{2-\beta}}}{\left\|\frac{u}{w^{\frac{2-\beta}{2}}}\right\|_{L^2(M)}^2}\,dx-\left(\int_{M}\frac{\frac{|u|^{\frac{2}{q}}}{\d^{\frac{2-\beta}{q}}}}{\left\|\frac{u}{w^{\frac{2-\beta}{2}}}\right\|_{L^2(M)}^{\frac{2}{q}}}\,dx \right)^{q} \\&
      \leq  \frac{(q-1)(Q-\beta)}{2-\beta} \log\left(\frac{C_{LH}\int_{M}|\nabla_{H}u|^2\,dx}{\left\|\frac{u}{w^{\frac{2-\beta}{2}}}\right\|_{L^2(M)}}\right)+J(u)\,,
      \end{split}
\end{equation}
where we have defined $J(u)$ by 
\[
J(u):= (q-1) \int_{M}\frac{\frac{|u|^2}{\d^{2-\beta}}}{\left\|\frac{u}{w^{\frac{2-\beta}{2}}}\right\|^{2}_{L^2(M)}}\log \left(\frac{|u|^2 \d^{(\beta-2)-(2Q-4)}}{\left\|\frac{u}{w^{\frac{2-\beta}{2}}}\right\|^{2}_{L^2(M)}}\right)\,dx\,,
\] 
and the proof is complete.
\end{proof}
\begin{rem}\label{rem.J(u)}
Notice that $J(u)$ becomes non-positive whenever for a.e. $x \in \G$ we have
    \[
   \left[ (\beta-2)-(2Q-4)\right]\log(w(x))\leq 0\,.
    \]
    Since for $0\leq \beta <2$ we have $(\beta-2)-(2Q-4)<0$, we conclude that for $w \geq 1$ such that $\log w \geq 0$ for a.e $x \in M$, we get that $J(u)\leq 0$,    
\end{rem}

\section{applications of the abstract results using known geometric hardy inequalities}\label{Sec4}

In this section, we employ the unified method of deriving logarithmic Hardy and Hardy-Poincar\'e inequalities that was developed in Section \ref{Sec3}.
\begin{cor}[Logarithmic geometric Hardy inequality on $\G^{+}$ ]\label{cor.Gplus}
    Let $\G^{+}$ be the half-space of a stratified group $\G$ of homogeneous dimension $Q\geq  3$. For $0\leq \beta < 2$, and for all $u \in C_{0}^{\infty}(\G^{+})$ we have    
     \begin{equation}\label{log.hardy}
\begin{split}
    &\int_{\G^{+}}\frac{\frac{|u|^{2}}{\dd^{2-\beta}}}{\left\|\frac{u}{\dx^{\frac{2-\beta}{2}}}\right\|^{2}_{L^2(\G^{+})}}\log\left(\frac{|u|^{2}\dd^{2Q-4}}{\left\|\frac{u}{\dx^{\frac{2-\beta}{2}}}\right\|^{2}_{L^2(\G^{+})}}\right)\,dx \\ &\leq\frac{Q-\beta}{2-\beta}\log\left(\frac{C_{LH,1}\int_{\G^{+}}|\nabla_{H}u|^2\,dx}{\left\|\frac{u}{\dx^{\frac{2-\beta}{2}}}\right\|^{2}_{L^2(\G^{+})}}\right)\,,
\end{split}
\end{equation}
where  
 \begin{equation}\label{CQb}
   C_{LH,1}:=A_{2}^{\frac{2Q(2-\beta)}{Q-\beta}}2^{\frac{2\beta(Q-2)}{Q-\beta}}\,.
   \end{equation}
 In particular, when $\left\|\frac{u}{\dx^{\frac{2-\beta}{2}}}\right\|_{L^2(\G^{+})}=1$, we have 
 
\begin{equation}\label{L1=1}
\begin{split}
    &  \int_{\G^{+}}\frac{|u|^{2}}{\dd^{2-\beta}}\log \left(|u| \dd^{Q-2} \right)\,dx \\& 
    \leq \frac{Q-\beta}{2(2-\beta)} \log \left(C_{LH,1}\|\nabla_{H}u\|^{2}_{L^2(\G^{+})}\right)\, 
\end{split}        
\end{equation}
\end{cor}
\begin{proof}
    The proof is an immediate consequence of Theorems \ref{thm.lh.rs1} and \ref{thm.logHardyM} if one chooses $w(x)=\dd$ and $M=\G^{+}$.
\end{proof}
\begin{cor}[Geometric Poincar\'e-Hardy inequality on $\G^{+}$. I]\label{cor12}
Let $\G^{+}$ be the half space of a stratified group $\G$ of homogeneous dimension $Q\geq  3$. For $0\leq \beta < 2$, $q>1$, and for all $u \in C_{0}^{\infty}(\G^{+})$ we have
    \begin{equation}\label{poin.hardy.G}
   \begin{split}
      &\int_{\G^{+}}\frac{\frac{|u|^2}{\dd^{2-\beta}}}{\left\|\frac{u}{\dx^{\frac{2-\beta}{2}}}\right\|_{L^2(\G^{+})}^2}\,dx-\left(\int_{\G^{+}}\frac{\frac{|u|^{\frac{2}{q}}}{\dd^{\frac{2-\beta}{q}}}}{\left\|\frac{u}{\dx^{\frac{2-\beta}{2}}}\right\|_{L^2(\G^{+})}^{\frac{2}{q}}}\,dx \right)^{q} \\&
      \leq  \frac{(q-1)(Q-\beta)}{2-\beta} \log\left(\frac{C_{LH,1}\int_{\G^{+}}|\nabla_{H}u|^2\,dx}{\left\|\frac{u}{\dx^{\frac{2-\beta}{2}}}\right\|_{L^2(\G^{+})}^2}\right)+J(u)\,,
      \end{split}
\end{equation}
where $C_{LH,1}$ is as in \eqref{CQb}, and we have defined $J(u)$ by 
\[
J(u):= (q-1) \int_{\G^{+}}\frac{\frac{|u|^2}{\dd^{2-\beta}}}{\left\|\frac{u}{\dx^{\frac{2-\beta}{2}}}\right\|^{2}_{L^2(\G^{+})}}\log \left(\frac{|u|^2 \dd^{(\beta-2)-(2Q-4)}}{\left\|\frac{u}{\dx^{\frac{2-\beta}{2}}}\right\|^{2}_{L^2(\G^{+})}}\right)\,dx\,.
\]  

\end{cor}
\begin{proof}
    The result is an immediate consequence of Theorem \ref{thm.poi.M} if one chooses $w(x)=\dd$ and $M=\G^{+}$.
\end{proof}

\begin{cor}[Geometric Hardy-Poincar\'e inequality on $\G^{+}$. II]\label{cor13}
Let $\G^{+}$ be the half space of a stratified group $\G$ of homogeneous dimension $Q\geq  3$. For $0\leq \beta < 2$, $q>1$, and for all $u \in C_{0}^{\infty}(\G^{+})$ we have
    \begin{equation}\label{poin.hardy.G}
   \begin{split}
      &\int_{\G^{+}}\frac{\frac{|u|^2}{(1+\dd)^{2-\beta}}}{\left\|\frac{u}{(1+\dx)^{\frac{2-\beta}{2}}}\right\|_{L^2(\G^{+})}^2}\,dx-\left(\int_{\G^{+}}\frac{\frac{|u|^{\frac{2}{q}}}{(1+\dd)^{\frac{2-\beta}{q}}}}{\left\|\frac{u}{(1+\dx)^{\frac{2-\beta}{2}}}\right\|_{L^2(\G^{+})^2}^{\frac{2}{q}}}\,dx \right)^{q} \\&
      \leq  \frac{(q-1)(Q-\beta)}{2-\beta} \log\left(\frac{C_{LH,1}\int_{\G^{+}}|\nabla_{H}u|^2\,dx}{\left\|\frac{u}{(1+\dx)^{\frac{2-\beta}{2}}}\right\|_{L^2(\G^{+})}^2}\right)\,,
      \end{split}
\end{equation}
where $C_{LH,1}$ is as in \eqref{CQb}.
    
\end{cor}
\begin{proof}
    The result is an immediate consequence of Theorem \ref{thm.poi.M} is one chooses $w(x)=\dd+1$ and $M=\G^{+}$, and taking  into account Remark \ref{rem.J(u)}.
\end{proof}

The following result is the logarithmic counterpart of the horizontal Hardy inequality on $\G$ as in Theorem \ref{thm.lh.rs2}.
\begin{cor}[Horizontal logarithmic Hardy inequality on $\G$]\label{cor.hor}
    Let $\G$ be as in Theorem \ref{thm.lh.rs2} and let $0\leq \beta \leq 2$. Then for all $u \in C_{0}^{\infty}(\G)$ we have 
     \begin{equation}\label{log.hardyG}
    \int_{\G}\frac{\frac{|u|^{2}}{|x'|^{2-\beta}}}{\left\|\frac{|u|}{|x'|^{\frac{2-\beta}{2}}}\right\|_{L^2(\G)}}\log\left(\frac{|u|^{2}|x'|^{2Q-4}}{\left\|\frac{|u|}{|x'|^{\frac{2-\beta}{2}}}\right\|_{L^2(\G)}}\right)\,dx \leq\frac{Q-\beta}{2-\beta}\log \left(\frac{C_{LH,2}\int_{\G}|\nabla_{H}u|^2\,dx }{\left\|\frac{|u|}{|x'|^{\frac{2-\beta}{2}}}\right\|_{L^2(\G)}}\right)\,,
\end{equation}
where  
 \begin{equation}\label{CQbG}
   C_{LH,2}:=A_{2}^{\frac{2Q(2-\beta)}{Q-\beta}}\left(\frac{N-2}{2} \right)^{-2\frac{\beta(Q-2)}{Q-\beta}}\,.
   \end{equation}
 In particular, when $\left\|\frac{u}{|x'|^{\frac{2-\beta}{2}}}\right\|_{L^2(\G)}=1$, then we have 
 
\begin{equation}\label{L1=1G}
        \int_{\G}\frac{|u|^{2}}{|x'|^{2-\beta}}\log \left(|u| |x'|^{Q-2} \right)\,dx
        \leq \frac{Q-\beta}{2(2-\beta)} \log \left(C_{LH,2}\|\nabla_{H}u\|^{2}_{L^2(\G)}\right)\,  
\end{equation}
\end{cor}
\begin{proof}
    The proof is an immediate consequence of Theorems \ref{thm.lh.rs2} amd \ref{thm.logHardyM} if one chooses $w(x)=|x'|.$
\end{proof}
\begin{cor}[Horizontal Hardy-Poincar\'e inequality on $\G$]\label{cor15}
Let $\G$ be a stratified group as in the hypothesis of Theorem \ref{thm.lh.rs2}. Then  for $0\leq \beta < 2$, $q>1$, and for all $u \in C_{0}^{\infty}(\G)$ we have
    \begin{equation}\label{poin.hardy.G}
    \begin{split}
          &\int_{\G}\frac{\frac{|u|^2}{|x'|^{2-\beta}}}{\left\|\frac{u}{|x'|^{\frac{2-\beta}{2}}}\right\|_{L^2(\G)}^2}\,dx-\left(\int_{\G}\frac{\frac{|u|^{\frac{2}{q}}}{|x'|^{\frac{2-\beta}{q}}}}{\left\|\frac{u}{|x'|^{\frac{2-\beta}{2}}}\right\|_{L^2(\G)^2}^{\frac{2}{q}}}\,dx \right)^{q} \\&
      \leq  \frac{(q-1)(Q-\beta)}{2-\beta} \log\left(\frac{C_{LH,2}\int_{\G}|\nabla_{H}u|^2\,dx}{\left\|\frac{u}{|x'|^{\frac{2-\beta}{2}}}\right\|_{L^2(\G)}^2}\right)+J(u)\,,
      \end{split}
\end{equation}
    where $C_{LH,2}$ is as in \eqref{CQbG}, and we have defined $J(u)$ by 
\[
J(u):= (q-1) \int_{\G}\frac{\frac{|u|^2}{|x'|^{2-\beta}}}{\left\|\frac{u}{|x'|^{\frac{2-\beta}{2}}}\right\|^{2}_{L^2(\G)}}\log \left(\frac{|u|^2 |x'|^{(\beta-2)-(2Q-4)}}{\left\|\frac{u}{|x'|^{\frac{2-\beta}{2}}}\right\|^{2}_{L^2(\G)}}\right)\,dx\,.
\]  
\end{cor}
\begin{proof}
    The result follows by Proposition \ref{thm.poi.M}.
\end{proof}

For the horizontal Hardy inequality \eqref{hardy2} we obtain its analogue with respect to a semi-Gaussian measure on the stratified group $\G$; i.e., a product of measure $\mu_1 \times \mu_2$, where $\mu_1$ is the Gaussian measure on the first stratum of (the Lie algebra of) $\G$, and $\mu_2$ is a Lebesgue measure. This type of horizontal Hardy inequality, given in the next theorem, can be called the ``horizontal Gross-type Hardy inequality'', since it serves as the analogue of the Gross's Logarithmic Sobolev inequality on $\R^n$ \cite{Gro75} with weights. The following proof is an adaptation of \cite[Theorem 7.1]{CKR21b} taking into account the presence of the weight.
\begin{thm}[Horizontal Gross-type Hardy inequality on $\G$]\label{thm.gaus.Hardy.G}
Let $\G$ be a stratified group with homogeneous dimension $Q \geq 3$, topological dimension $n$, and let $N$ be the dimension of the first stratum of its Lie algebra, i.e., for $x \in \G$ we can write $x=(x',x'') \in \mathbb{R}^{N} \times \mathbb{R}^{n-N}$. 
Then the following weighted ``semi-Gaussian'' geometric logarithmic Hardy inequality is satisfied
\begin{equation}
    \label{gaus.log.Hard}
    \int_{\G}\frac{|g(x)|^2}{|x'|^{2-\beta}} \log \left(|x'|^{Q-2}|g(x)| \right)\,d\mu \leq \int_{\G} |\nabla_{H}g|^2\,d\mu\,,
\end{equation}
for all $0\leq \beta <2$, and for all $g$ such that $\left\|\frac{g}{|x'|^{\frac{2-\beta}{2}}} \right\|_{L^2(\mu)}=1,$
 where $\mu=\mu_1 \otimes \mu_2$, and $\mu_1$ is the Gaussian measure on $\mathbb{R}^{N}$ given by $d\mu_1=\gamma e^{-\frac{|x'|^2}{2}}dx'$, for $x'\in \mathbb{R}^{N}$, where the normalisation constant $\gamma$ is given by 
 \begin{equation}
     \label{k-w}
     \gamma:= 
     \left(\frac{Q-\beta}{2(2-\beta)}C_{LH,2}e^{\left(\frac{N}{2}+\frac{1}{4} \right)\left(\frac{2(2-\beta)}{Q-\beta} \right)-1} \right)^{\frac{Q-\beta}{2-\beta}}
 \end{equation} 
 and $\mu_2$ is the Lebesgue measure $dx''$ on $\mathbb{R}^{n-N}$, and the constant $C_{LH,2}$ is given by \eqref{CQbG}.
\end{thm}

\begin{proof}
    Assume that $g$ is such that $\left\| \frac{g}{|x'|^{\frac{2-\beta}{2}}}\right\|_{L^2(\mu)}=1$ , where $\mu$ is as in the hypothesis. Define $u$ by 
\[
u(x)=\gamma^{1/2}e^{-\frac{|x'|^2}{4}}g(x)\,,
\]
for $\gamma$ as in \eqref{k-w}. Then we have \[
1=\int_{\G}\frac{|g(x)|^2}{|x'|^{2-\beta}}\,d\mu=\int_{\G}\gamma^{-1}e^{\frac{|x'|^2}{2}}\frac{|u(x)|^2}{|x'|^{2-\beta}}\,d\mu=\int_{\G}\frac{|u(x)|^2}{|x'|^{2-\beta}}\,dx\,,
\]
i.e. we have proved that $\left\|\frac{u}{|x'|^{\frac{2-\beta}{2}}} \right\|_{L^2(\G)}=1$.
An application of the logarithmic Hardy  inequality \eqref{L1=1G}, together with the properties of the logarithm, give
\begin{equation}
\label{log.Hold,u}
\begin{split}
& \int_{\G}\frac{|g(x)|^2}{|x'|^{2-\beta}} \log \left(|x'|^{Q-2}|g(x)| \right)\,d\mu \\& \leq  \int_{\G}\frac{|u(x)|^2}{|x'|^{2-\beta}} \log\left(\gamma^{-1/2}e^{\frac{|x'|^2}{4}}|x'|^{Q-2}u(x) \right)\,dx\nonumber\\ &
\leq\frac{Q-\beta}{2(2-\beta)} \log\left(C_{LH,2} \int_{\G}|\nabla_H u(x)|^2\,dx \right)\nonumber\\ &
+\log(\gamma^{-1/2})+\int_{\G}\frac{|x'|^2}{4}\frac{|u(x)|^2}{|x'|^{2-\beta}}\,dx\,.
\end{split}
\end{equation}
Note that we also have
\begin{multline}\label{EQ:aterm}
\int_{\G}\frac{|x'|^2}{4}\frac{|u(x)|^2}{|x'|^{2-\beta}}\,dx=\frac{1}{4} \int_{\G}|x'|^\beta |u(x)|^2\,dx=\int_{|x'|\leq 1} +\int_{|x'|\geq 1}  \\ \leq
 \frac{1}{4}\int_{|x'|\leq 1} \frac{|u(x)|^2}{|x'|^{2-\beta}}\,dx+
 \int_{|x'|\geq 1}
 \frac{|x'|^2}{4} |u(x)|^2\,dx
 \\ \leq \frac{1}{4}+\int_\G  \frac{|x'|^2}{4} |u(x)|^2\,dx\,,
\end{multline}
where we have used that $\beta<2$, and that $\left\|\frac{u}{|x'|^{\frac{2-\beta}{2}}} \right\|_{L^2(\G)}=1$. Now, to estimate the length 
\[
|\nabla_H g(x)|=\sqrt{\sum_{i=1}^{N}|X_ig(x)|^2}\,,
\]
first recall (see e.g. \cite{FR16}) that the vector fields $X_i$ for $i=1,\ldots,N$, are given by 
\[
X_i=\partial_{x_{i}^{'}}+\sum_{j=1}^{n-N}p_{j}^{i}(x')\partial_{x''_{j}}\,.
\]
For $i=1,\ldots,N$ we have
\begin{eqnarray}\label{Xg}
|X_ig(x)|^2 &=& \gamma^{-1} e^{\frac{|x'|^2}{2}} \left|X_iu(x)+\frac{x'_{i}}{2}u(x) \right|^{2}\nonumber\\
&=& \gamma^{-1} e^{\frac{|x'|^2}{2}} \left(|X_iu(x)|^2+\frac{(x'_{i})^2}{4}|u(x)|^2+{\rm Re} \overline{(X_iu(x))}x'_{i}u(x) \right)\,.
\end{eqnarray}
Moreover, for $x'_{i}$, $i=1,\ldots,N$, using integration by parts we get
\begin{eqnarray*}
{\rm Re}\int_{\G}\overline{(\partial_{x'_{i}}u(x))}x'_{i}u(x)\,dx & =& -{\rm Re}\int_{\G}(\partial_{x'_{i}}u(x))x'_{i}\overline{u(x)}\,dx-\int_{\G}|u(x)|^2\,dx\\
&=& -{\rm Re}\int_{\G}\overline{(\partial_{x'_{i}}u(x))}x'_{i}u(x)\,dx-1\,,
\end{eqnarray*}
that is we have
\begin{equation}
    \label{int.parts1}
    {\rm Re}\int_{\G}\overline{(\partial_{x'_{i}}u(x))}x'_{i}u(x)\,dx=-\frac{1}{2}\,.
\end{equation}
Applying similar arguments, for $j=1,\ldots,n-N$ and for $i=1,\ldots,N$, we get
\begin{eqnarray*}
   {\rm Re} \int_{\G}p_{j}(x')\overline{(\partial_{x_{j}''}u(x))}x'_{i}u(x)& = & - {\rm Re}\int_{\G}\partial_{x_{j}''}((p_{j}(x')u(x)x'_{i})\overline{u(x)}\,dx\\
    &=&- {\rm Re}\int_{\G}p_{j}(x')\overline{(\partial_{x_{j}''}u(x))}x'_{i}u(x)\,dx\,,
\end{eqnarray*}
or equivalently
\begin{equation}  \label{int.parts2}
   {\rm Re} \int_{\G}p_{j}(x')\overline{(\partial_{x_{j}''}u(x))}x'_{i}u(x)\,dx=0\,.
\end{equation}
Combining \eqref{int.parts1} and \eqref{int.parts2} we arrive at 
\[
\int_{\G}{\rm Re}\overline{(X_iu(x))}x'_{i}u(x)\,dx=-\frac{1}{2}\,, \quad \forall i=1,\ldots,n_1\,.
\]
Plugging \eqref{int.parts1} and \eqref{int.parts2} into \eqref{Xg} we get the 
\begin{equation}\label{thm.eq.nablag}
\int_{\G}|\nabla_{H}g(x)|^2\,d\mu=\int_{\G}|\nabla_{H}u(x)|^2\,dx+\int_{\G}\frac{|x'|^2}{4}|u(x)|^2\,dx-\frac{N}{2}\,.
\end{equation}
Combining inequalities \eqref{log.Hold,u},  and \eqref{thm.eq.nablag}, we see that to prove the desired inequality \eqref{gaus.log.Hard} it is enough to show that 
\[
\frac{Q-\beta}{2(2-\beta)} \log\left(C_{LH,2} \int_{\G} |\nabla_H u(x)|^2\,dx \right)+\log (\gamma^{-1/2})+\frac{1}{4} \leq \int_{\G} |\nabla_H u(x)|^2\,dx-\frac{N}{2}\,,
\]
which can be equivalently written as 
\[
\log\left(C_{LH,2}\gamma^{-\frac{2-\beta}{Q-\beta}}e^{\left(\frac{N}{2}+\frac{1}{4}\right)\left(\frac{2(2-\beta)}{Q-\beta} \right)}\int_{\G}|\nabla_H u(x)|^2\,dx \right) \leq \frac{2(2-\beta)}{Q-\beta}\int_{\G}|\nabla_H u(x)|^2\,dx\,.
\] 
Now, since $\log r \leq r-1$, for all $r>0$, it suffices to show that 
\[
e^{-1}C_{LH,2}\gamma^{-\frac{2-\beta}{Q-\beta}}e^{\left(\frac{N}{2}+\frac{1}{4}\right)\left(\frac{2(2-\beta)}{Q-\beta} \right)}\int_{\G}|\nabla_H u(x)|^2\,dx \leq \frac{2(2-\beta)}{Q-\beta} \int_{\G} |\nabla_H u(x)|^2\,dx\,,
\]
where the last is satisfied as an equality for $\gamma$ as in \eqref{k-w}, and we have
\begin{eqnarray*}
&&\log\left( C_{LH,2}\gamma^{-\frac{2-\beta}{Q-\beta}}e^{\left(\frac{N}{2}+\frac{1}{4}\right)\left(\frac{2(2-\beta)}{Q-\beta} \right)}\int_{\G}|\nabla_H u(x)|^2\,dx\right)\\
&=&\log\left(e^{-1}C_{LH,2}\gamma^{-\frac{2-\beta}{Q-\beta}}e^{\left(\frac{N}{2}+\frac{1}{2}\right)\left(\frac{2(2-\beta)}{Q-\beta} \right)}\int_{\G}|\nabla_H u(x)|^2\,dx \right)+1\\
&\leq & \frac{2(2-\beta)}{Q-\beta}\int_{\G}|\nabla_H u(x)|^2\,dx\,,
\end{eqnarray*}
completing the proof.
\end{proof}
\begin{rem}
Recall that in  \cite[Theorem 5.2]{CKR21a} the authors proved the following inequality 
\begin{equation}\label{Gh.CKR}
 \int_{\G}\frac{|g(x)|^2}{N(x)^{2-\beta}} \log \left(N(x)^{\frac{(Q-2)(1-\beta)}{2-\beta}}|g(x)| \right)\,d\mu \leq \int_{\G} |\nabla_{H}g|^2\,d\mu\,,
\end{equation}
where $\mu$ as as in Theorem \ref{thm.gaus.Hardy.G}, but with a normalisation constant $\gamma$ depending on a constant $M$, which in turn depends on the quasi-norm $N$ and on $\G$, such that $|x'|\leq M N(x)$, for all $x \in \G$. We note that, in contrast to the hypothesis of  \cite[Theorem 5.2]{CKR21a}, in Theorem \ref{thm.gaus.Hardy.G} the quasi-norm $N$ in this case, is a norm only on the first stratum, and the normalisation constant does not depend on the constant $M$. However, both inequalities \eqref{Gh.CKR} and \eqref{gaus.log.Hard}, when regarded in the Euclidean setting $\G=\R^n$, can coincide if one considers $N(x)=|x|$, where $|x|$ stands for the Euclidean norm of $x \in \R^n$, since in this case we also have $|x'|=|x|$. 
\end{rem}
 As the next result shows, on a stratified group $\G$ the analogue of the Horizontal Hardy-Poincar\'e inequality with respect to the semi-Gaussian measure $\mu$ holds true.
\begin{cor}[Horizontal Gross-type Hardy-Poincar\'e inequality on $\G$. I]\label{cor.hor.poi} Let $\G$ and $\mu$ be as in the hypothesis of Theorem \ref{thm.gaus.Hardy.G}. Then for  all $0\leq \beta <2$, and for all  $g \in C_{0}^{\infty}(\G)$, we have
\begin{equation}\label{hor.poi}
 \int_{\G}\frac{|g|^2}{|x'|^{2-\beta}}\,d\mu-\left(\int_{\G}\frac{|g|^{\frac{2}{q}}}{|x'|^{\frac{2-\beta}{q}}}\,d\mu\right)^q \leq 2(q-1)\int_{\G}|\nabla_{H}g|^2\,d\mu+J(g)\,,
\end{equation}
where we have set
   \[
J(g):=2(q-1)\int_{\G}\frac{|g|^2}{|x'|^{2-\beta}}\log \left(|g||x'|^{^{\frac{\beta-2}{2}-(Q-2)}}\left\|\frac{g}{|x'|^{\frac{2-\beta}{2}}} \right\|_{L^2(\mu)}^{-1}\right)\,d\mu\,.
\]
\end{cor}
\begin{proof}
For $u$ such that $\|u\|_{L^2(\G)}=1$, let us consider the function $b=b(q):=q \log \left(\int_{\G}|u|^{\frac{2}{q}}\,d\mu \right)$. Following the line of arguments developed in the proof of Theorem \ref{thm.poi.M} we get that 
\[
\int_{\G}|u|^2\,d\mu-\left(\int_{\G}|u|^{\frac{2}{q}}\right)^q\leq 2(q-1)\int_{\G}|u|^2\log |u|\,d\mu \,.
\]
Now, if we replace $u(x)$ by $\frac{g(x)}{|x'|^{\frac{2-\beta}{2}}}\left\|\frac{g}{|x'|^{\frac{2-\beta}{2}}} \right\|_{L^2(\mu)}^{-1}$, the latter inequality becomes 
\begin{equation*}    
\begin{split}
        &\ \int_{\G}\frac{|g|^2}{|x'|^{2-\beta}}\left\|\frac{g}{|x'|^{\frac{2-\beta}{2}}} \right\|_{L^2(\mu)}^{-2}\,d\mu   -\left(\int_{\G}\frac{|g|^{\frac{2}{q}}}{|x'|^{\frac{2-\beta}{q}}}\left\|\frac{g}{|x'|^{\frac{2-\beta}{2}}} \right\|_{L^2(\mu)}^{-\frac{2}{q}}\,d\mu\right)^q \\&
      \leq 2(q-1)\int_{\G}\frac{|g|^2}{|x'|^{2-\beta}}\left\|\frac{g}{|x'|^{\frac{2-\beta}{2}}} \right\|_{L^2(\mu)}^{-2}\log \left(|g||x'|^{\frac{\beta-2}{2}}\left\|\frac{g}{|x'|^{\frac{2-\beta}{2}}} \right\|_{L^2(\mu)}^{-1}\right)\,d\mu \,.  
      \end{split}
\end{equation*}
Now, we can write
\begin{equation*}
    \begin{split}
       &\ \int_{\G}\frac{|g|^2}{|x'|^{2-\beta}}\log \left(|g||x'|^{\frac{\beta-2}{2}}\left\|\frac{g}{|x'|^{\frac{2-\beta}{2}}} \right\|_{L^2(\mu)}^{-1}\right)\,d\mu \\&
       = \int_{\G}\frac{|g|^2}{|x'|^{2-\beta}}\log \left(|g||x'|^{Q-2}\left\|\frac{g}{|x'|^{\frac{2-\beta}{2}}} \right\|_{L^2(\mu)}^{-1}\right)\,d\mu \\&
       + \int_{\G}\frac{|g|^2}{|x'|^{2-\beta}}\log \left(|g||x'|^{\frac{\beta-2}{2}-(Q-2)}\left\|\frac{g}{|x'|^{\frac{2-\beta}{2}}} \right\|_{L^2(\mu)}^{-1}\right)\,d\mu:=I(g)+J(g)\,.
    \end{split}
\end{equation*}
Using the horizontal Gross-type Hardy inequality \eqref{gaus.log.Hard} to estimate $I(g)$ we get 
\begin{equation*}
    \begin{split}
       &\  \int_{\G}\frac{|g|^2}{|x'|^{2-\beta}}\left\|\frac{g}{|x'|^{\frac{2-\beta}{2}}} \right\|_{L^2(\mu)}^{-2}\,d\mu-\left(\int_{\G}\frac{|g|^{\frac{2}{q}}}{|x'|^{\frac{2-\beta}{q}}}\left\|\frac{g}{|x'|^{\frac{2-\beta}{2}}} \right\|_{L^2(\mu)}^{-\frac{2}{q}}\,d\mu\right)^q \\&
       \leq 2(q-1)\left\|\frac{g}{|x'|^{\frac{2-\beta}{2}}} \right\|_{L^2(\mu)}^{-2} \int_{\G}|\nabla_{H}g|^2\,d\mu + J(g)\,,
    \end{split}
\end{equation*}
or, after simplification,
\begin{equation*}
    \int_{\G}\frac{|g|^2}{|x'|^{2-\beta}}\,d\mu-\left(\int_{\G}\frac{|g|^{\frac{2}{q}}}{|x'|^{\frac{2-\beta}{q}}}\,d\mu\right)^q \leq 2(q-1)\int_{\G}|\nabla_{H}g|^2\,d\mu+J(g)
\end{equation*}
where we have set
\begin{equation*}\label{def.J(g)}
J(g):=2(q-1)\int_{\G}\frac{|g|^2}{|x'|^{2-\beta}}\log \left(|g||x'|^{^{\frac{\beta-2}{2}-(Q-2)}}\left\|\frac{g}{|x'|^{\frac{2-\beta}{2}}} \right\|_{L^2(\mu)}^{-1}\right)\,d\mu\,,
\end{equation*}
and the proof is complete.
\end{proof}
Corollary \ref{cor.hor.poi} together with Remark \ref{rem.J(u)} imply the following:
\begin{cor}[Horizontal Gross-type Hardy-Poincar\'e inequality on $\G$. II]\label{cor.hor.poi2} Let $\G$ and $\mu$ be as in the hypothesis of Theorem \ref{thm.gaus.Hardy.G}. Then for  all $0\leq \beta <2<Q$, and for all  $g \in C_{0}^{\infty}(\G)$, we have
\begin{equation}\label{hor.poi2}
 \int_{\G}\frac{|g|^2}{(|x'|+c)^{2-\beta}}\,d\mu-\left(\int_{\G}\frac{|g|^{\frac{2}{q}}}{(|x'|+c)^{\frac{2-\beta}{q}}}\,d\mu\right)^q \leq 2(q-1)\int_{\G}|\nabla_{H}g|^2\,d\mu\,.
\end{equation}
for any $c \geq 1$.
\end{cor}

\begin{rem}
    It is important to note that the horizontal Gross-type Hardy-Poincar\'e inequalities \eqref{hor.poi} and \eqref{hor.poi2} are already new in the Euclidean setting, as they generalise the so-called generalised Poincar\'e inequality \eqref{Poi.bec}  by adding weights.
    
\end{rem}
Furthermore, using Corollary \ref{cor.hor.poi2} we can generalise the ``standard'' Poincar\'e inequality \eqref{Poi} in the Euclidean case $\R^n$, $n \geq 3$, by adding weights.
\begin{cor}[Poincar\'e inequality with weights on $\R^n$]\label{cor1.poi}
    Let $\R^n$ with $n\geq 3$. Then for $g \in C_{0}^{\infty}(\R^n)$, $0 \leq \beta <2$, and $c\geq 1$, we have
    \begin{equation}
        \label{poincare.weights}
            \int_{\R^n}\left(\frac{|g|}{(|x|+c)^{\frac{2-\beta}{2}}}-\mu\left(\frac{|g|}{(|x|+c)^{\frac{2-\beta}{2}}}\right)^2\right)d\mu\, \leq 2 \int_{\R^n}|\nabla g|^2\,d\mu\,,
    \end{equation}
    where $|\cdot|$ stands for the Euclidean distance in $\R^n$, $\mu$ for the Gaussian measure on $\R^n$, and we have used the notation $\mu(f):=\int_{\R^n}f\,d\mu$.
\end{cor}
\begin{proof} Observe that in the case where $\G=\R^n$ we have $|x'|=|x|$, and the measure $\mu$ as in the hypothesis of Theorem \ref{thm.gaus.Hardy.G} is simply the Gaussian measure on $\R^n$. Hence, as immediate consequence of Corollary \ref{cor.hor.poi2} we get 
\[
 \int_{\R^n}\frac{|g|^2}{(|x|+c)^{2-\beta}}-\left(\int_{\R^n}\frac{|g|}{(|x|+c)^{\frac{2-\beta}{2}}}\right)^2\,d\mu \leq 2 \int_{\R^n}|\nabla g|^2\,d\mu\,.
\]
The proof is complete under the observation that for the probability measure $\mu$ the left-had side of the above inequality is equal to the left-hand side of \eqref{poincare.weights}.
\end{proof}

\section{Acknowledgement}
It is a pleasure to thank Gerassimos Barbatis and Michael Ruzhansky for suggesting the problem studied here, and for their helpful comments and suggestions. The author is a postdoctoral fellow of the Research Foundation-Flanders (FWO) under the postdoctoral grant No 12B1223N. 

\end{document}